\documentclass[a4paper]{article}
\usepackage[in]{fullpage}
\usepackage[final]{graphicx}
\usepackage{amsmath,amsthm,amsfonts}
\usepackage{url}
\usepackage{xspace}
\usepackage[numbers]{natbib}

\title{On harmonic analysis of vector-valued signals}
\author{Stephen J. Sangwine\thanks{Stephen J. Sangwine is with the
        School of Computer Science and Electronic Engineering,
        University of Essex, Wivenhoe Park, Colchester, CO4 3SQ, UK.
        Email: \protect\url{sjs@essex.ac.uk}.}}

\providecommand*{\abbreviation}[1]{\textit{#1}}
\providecommand*{\etc}{\abbreviation{etc.}\xspace}
\providecommand*{\ie}{\abbreviation{i.e.}\xspace}

\newtheorem{theorem}{Theorem}
\newtheorem{definition}{Definition}

\renewcommand{\vector}[1]{\ensuremath{\boldsymbol{#1}}\xspace}

\newcommand{\inner}[2]{\ensuremath{\left\langle #1, #2\right\rangle}\xspace}
\newcommand{\norm}[1]{\ensuremath{\left\lVert#1\right\rVert}\xspace}
\newcommand{\modulus}[1]{\ensuremath{\left\lvert#1\right\rvert}\xspace}
\newcommand*{\R}{\ensuremath{\mathbb{R}}\xspace}
\newcommand{\I}{\ensuremath{i}\xspace}

\begin{document}
\maketitle

\begin{abstract}
A vector-valued signal in $N$ dimensions is a signal whose value
at any time instant
is an $N$-dimensional vector, that is, an element of $\R^N$.
The sum of an arbitrary number of such signals \emph{of the same frequency} is shown
to trace an ellipse in $N$-dimensional space, that is, to be \emph{confined to a plane}.
The parameters of the ellipse (major and minor axes, represented by $N$-dimensional vectors;
and phase) are obtained algebraically in terms of the directions of oscillation of the constituent
signals, and their phases.
It is shown that the major axis of the ellipse can always be determined algebraically.
That is, a vector, whose value can be computed algebraically
(without decisions or comparisons of magnitude)
from parameters of the constituent signals,
always represents the major axis of the ellipse.
The ramifications of this result for the processing and Fourier analysis of signals with vector
values or samples are discussed, with reference to the definition of Fourier
transforms, particularly discrete Fourier transforms, such as have been defined
in several hypercomplex algebras, including Clifford algebras.
The treatment in the paper, however, is entirely based on signals with values in
$\R^N$.
Although the paper is written in terms of vector \emph{signals}
(which are taken to include images and volumetric images),
the analysis clearly also applies to a superposition of
simple harmonic motions in $N$ dimensions.
\end{abstract}

\section{Introduction}
This paper is concerned with the analysis of vector-valued signals,
in particular,
it is concerned with harmonic analysis of such signals in terms of sinusoidally
varying frequency components.
The paper employs straightforward mathematics,
using $N$-dimensional vectors,
and the concepts of \emph{norm} and \emph{inner product},
and elementary trigonometric functions.
The results presented are not claimed to be profound,
nevertheless, they do not appear to be available in the literature,
to the knowledge of the author.
\subsection{Vector-valued signals}
A real-valued signal $f(t)$ has a value at each instant in time, $t$,
which is an element of \R, the set of real numbers.
In practice, such a signal may be discretised in time
and in amplitude (this is said to be a digital signal).
In this paper, we consider the mathematics of signals without concern
for discretisation in time or amplitude, but it should be understood that the
theory presented here, although expressed in terms of
continuous time/continuous amplitude signals,
is not invalidated by discretisation.

A vector-valued signal has values at each time instant which are $N$-dimen\-sional
vectors, that is elements of $\R^N$.
We place no restriction on $N$, which can be any positive integer
(in the case $N=1$ of course, the results reduce to the classical
case of real-valued signals, and for $N=2$, the results are valid
for signals with complex values, although we express them in the
paper using vectors in $\R^2$).

Although we express the results in the paper in terms of signals which
are functions of time (for simplicity) the results are, of course, valid for signals
which are functions of some other variable, and in particular, the
results may be simply extended to images
(or volumetric images)
with pixels
(or voxels)
which have values in $\R^N$.

At the risk of overstressing the point, we emphasise that the \emph{values}
of the signal are vectors.
The signal itself is a time-series of vector values.
We are not discussing in this paper vectors containing a whole signal,
nor by the dimensionality $N$ do we mean the number of dimensions
of the signal (1 for a time series, 2 for images, 3 for volumetric images, \etc).
This point was discussed at greater length in \cite{10.1109/TSP.2003.812734}
to which we refer the reader (see particularly Table 1, where the concept
of vector as presented in the current paper corresponds to the column on
the left of the table, labelled `Components per sample').
\subsection{Frequency components}
The central concern of this paper is to understand the meaning of the
concept of a \emph{frequency component} of a vector-valued signal.
The classical theory of Fourier analysis applied to real-valued signals
expresses a signal in terms of a summation of sinusoidal signals
of various frequencies which
are scaled in amplitude and shifted in time (phase shifted).
(We are omitting here, for simplicity,
the differences between Fourier series,
Fourier integrals or transforms, and discrete Fourier transforms,
because these differences are not relevant to the discussion that
follows.)
In the case of signals with real values,
Fourier analysis represents a signal by a sum of scaled and shifted
pairs of complex conjugate exponentials
(representing positive and negative frequencies).
The sum of these scaled and shifted exponentials reconstructs the signal
which has been analysed, the imaginary parts cancelling out to yield a
real result.
Once we move from a real-valued signal to a complex-valued signal,
the picture becomes more complicated, and less well presented in
the literature.
However, the same idea applies in this case, except that the pair
of complex exponentials are no longer complex conjugates,
and when added,
their imaginary parts do not cancel out.
The results in this paper show that it is not difficult
to understand how a complex signal may be represented in terms of
frequency components.
Beyond $N=2$, the picture is less clear.

We show that the concept of a \emph{frequency component} can be
simply understood in terms of an elliptical path through the space
of the $N$-dimensional signal values traversed $\nu$ times per second,
where $\nu$ is the frequency of the component in hertz.
The ellipses are composed of the sum of a cosinusoidal and a sinusoidal
oscillation \emph{in a plane}, regardless of the value of $N$
(excluding $N=1$ of course).
A canonic decomposition of this ellipse into \emph{orthogonal} cosinusoidal
and sinusoidal components is given.
This seems to the author a remarkable result, because it means that
the geometric interpretation of the concept of \emph{frequency component}
does not change as one increases the dimensionality of the signal values.
This in turn means that harmonic analysis into sinusoidal frequency
components is essentially of the same character regardless of $N$,
and this has significant ramifications for the construction of Fourier transforms
of vector-valued signals,
which we discuss in \S\,\ref{sec:ramifications} of the paper.
\subsection{Polarization}
A signal with samples in two or more dimensions may be \emph{polarized}.
The concept of polarization is well-known in physics,
particularly for electromagnetic waves including light,
and seismic waves propagating through rock.
For a discussion of signal polarization
(for the complex, or two-dimensional, case only),
see \cite{10.1016/j.sigpro.2006.03.019,10.1109/TSP.2008.925961}.

It will become clear in this paper that every frequency component
of a vector signal as discussed in this paper is polarized,
because it is confined to a plane in the $N$ dimensional space of the signal values.
Special cases are \emph{linear polarization}
(the signal values oscillate along a line in $N$-dimensional space);
\emph{circular polarization}
(the signal values oscillate around a circular path in $N$-dimensional space);
\emph{elliptical polarization}
(the general case: the signal values oscillate around an elliptical path in $N$-dimensional space).
In the second and third cases, there is also a direction of polarization
(the sense in which the values traverse the circle/ellipse: clockwise or anti-clockwise).

The question of whether a \emph{signal}
(as opposed to a \emph{single frequency component} of a signal)
is polarized is outside the scope of this paper.
\subsection{Vector sensors}
A vector signal may be captured by a vector sensor,
that is a sensor with the ability to capture orthogonal components
of an incident wave.
Examples include vector geophones,
3-axis accelerometers, and gyroscopes.
Since the physical world is 3-dimensional,
vector signals with 3-dimensions are more common
than those with higher dimensions.
Notice that we do not consider waves: once a wave impinges on a vector
sensor, information about the wave is reduced to oscillation in $N$
dimensions as detected by the sensor.
To capture more information from the wave, for example to permit estimation
of the direction of arrival, it is necessary to utilise an array of sensors,
which is outside the scope of this paper.
\subsection{Prior work}
The ideas presented in this paper have been developed from some prior
work along similar lines for the case of signals with complex values.
Research by Andrew McCabe, Terry Caelli and their co-workers in the late 1990s
\cite{McCabe:2000,10.1007/3-540-44732-6}
showed how harmonic analysis of images with complex pixels could be
understood in terms of elliptical paths in the complex plane
(they called this idea spatiochromatic
image analysis because their complex pixels represented chrominance,
the aspect of a colour image that represents colour).
This idea was extended to three-dimensions (again in the context of colour images)
in a 2007 paper by Todd Ell and the author \cite{10.1109/TIP.2006.884955}.
This paper analysed quaternion Fourier transforms of colour images and showed
that the Fourier domain representation of the image consists of elliptical paths
through the space of the pixel values (colour space).
The present paper extends these ideas to $N$ dimensions,
finds the ellipse parameters,
and shows that the same mathematics applies in any number of dimensions,
not just 2, 3, or 4.
\subsection{Notation}
Throughout this paper we use the following notations:
\begin{itemize}
\item vectors (in an arbitrary number of dimensions, $N>1$) are indicated in bold type,
      as \vector{u};
\item the modulus of a vector is indicated as \modulus{\vector{u}},
      and is the magnitude or length of the vector
      (the square root of the sum of the squared Cartesian components);
\item the inner product of two vectors is indicated as \inner{\vector{u}}{\vector{v}}.
      It is equal to $\modulus{\vector{u}}\,\modulus{\vector{v}}\cos\theta$, where
      $\theta$ is the angle between the two vectors;
\item the norm of a vector is indicated as $\norm{\vector{u}}$, and means the
      square of the magnitude of the vector, that is $\norm{\vector{u}}=\modulus{\vector{u}}^2$.
      The norm may also be computed as $\norm{\vector{u}}=\inner{\vector{u}}{\vector{u}}$.
\end{itemize} 
\section{Elliptical paths}
\label{sec:ellipses}
In this section,
we show that a sum of sinusoidal signals
\emph{of the same frequency}
in $N$ dimensions
takes the form of an elliptical path through the space of the signal values,
that is, the signal values are confined to a plane.
This result holds no matter how many such signals are added,
and irrespective of their relative amplitudes, phases,
and orientations in $N$-dimensional space.
An alternative interpretation is that the fundamental definition of
oscillation at a given frequency, in $N$ dimensions, is an oscillation
along an elliptical path in the space of the signal values,
traversed $\nu$ times per second for a signal with a frequency of $\nu$\,Hz.

We first define what we mean by a sinusoidal signal in an $N$-dimensional Euclidean space.
\begin{definition}
\label{def:shm-nd}
A sinusoidal signal in $N$ dimensions may be represented in the form:
$f(t) = \vector{n}\sin\left(\omega t + \phi\right)$,
where $\vector{n}$ is a vector in the $N$-dimensional space\footnote{In Cartesian coordinates,
for example, $\vector{n}=(n_1, n_2, n_3, \ldots, n_i, \ldots, n_N), n_i\in\R$.},
$\omega$ is the angular frequency of the sinusoid,
and $\phi$ is an initial phase at $t=0$.
$\omega$ and $\phi$ are real.
\end{definition}
The signal $f(t)$ clearly oscillates sinusoidally along a line in $N$-dimensional
space defined by $\vector{n}$.
Note that \vector{n} is not necessarily a unit vector,
hence the amplitude of the oscillation is represented by the modulus of \vector{n}.
A single signal of this form is referred to as \emph{linearly polarized}.
A more general case, however, and the central concern of this paper,
is a superposition of an arbitrary number of signals taking the same form,
but with different parameters
\emph{apart from the frequency},
in particular,
with different directions and amplitudes of oscillation in $N$-dimensional space,
and different initial phases.
We show in Theorem \ref{thm:planar} that such a superposition yields,
in general,
a signal that traces over time an elliptical path \emph{in a plane}
in $N$-dimensional space,
regardless of the value of $N$.
In Theorem \ref{thm:major_minor} we give a parameterisation of the ellipse in terms
of major and minor axes.
\begin{theorem}
\label{thm:planar}
The sum of an arbitrary number, $K$, of signals
as defined in Definition~\ref{def:shm-nd},
with differing amplitudes, phases, and directions of oscillation,
\emph{but the same angular frequency}:
\begin{align}
f(t) &= \sum_{i=1}^{K}\vector{n}_i \sin\left(\omega t + \phi_i\right)\label{eqn:sum}
\intertext{may be expressed as:}
f(t) &= \vector{c}\sin(\omega t) + \vector{s}\cos(\omega t)\label{eqn:sum_cs}
\end{align}
where \vector{c} and \vector{s} are vectors given by:
\[
\vector{c} = \sum_{i=1}^{K}\vector{n}_i\cos\phi_i,\qquad
\vector{s} = \sum_{i=1}^{K}\vector{n}_i\sin\phi_i.
\]
Since the sum oscillates in two directions only,
it always lies within a plane in $N$-dimensional space.
\end{theorem}
\begin{proof}
Expanding the sine function
in \eqref{eqn:sum} we obtain:
\begin{align}
f(t) &= \sum_{i=1}^{K} \vector{n}_i
        \left[
        \sin(\omega t)\cos\phi_i +
        \cos(\omega t)\sin\phi_i
        \right]\notag\\
     &= \sum_{i=1}^{K}\vector{n}_i\cos\phi_i\sin(\omega t) +
        \sum_{i=1}^{K}\vector{n}_i\sin\phi_i\cos(\omega t)\notag
\end{align}
which gives \eqref{eqn:sum_cs}.
\end{proof}

Notice that,
in this formulation,
the value of the signal at $t=0$ is given by the vector \vector{s}.

Notice that we now know \vector{c} and \vector{s} in terms of the
vectors and phases that define the sum.
Therefore we also know their moduli,
and the cosine of the angle between them from the inner product.
Note that, in general, \vector{c} and \vector{s} are not orthogonal.
Notice also that the two vectors in this result are
\emph{not} necessarily the major and minor axes of the ellipse,
and although they are sufficient to parameterise the ellipse,
they are not a convenient parameterisation.
A better parameterisation would use two vectors aligned along
the major and minor axes and this is what we seek next,
since it expresses the oscillation in the plane in terms
of two \emph{perpendicular} oscillations,
such that one of them has the largest possible amplitude.
Let the vector \vector{a} be aligned along the major axis of the ellipse,
and the vector \vector{b} be aligned perpendicular to \vector{a} along the minor axis.
We must find \vector{a} and \vector{b} in terms of \vector{c} and \vector{s}.

\begin{theorem}
\label{thm:major_minor}
The sum of an arbitrary number, $K$, of signals as defined in Theorem \ref{thm:planar}
may be expressed in terms of the major and minor axes of the ellipse as:
\begin{equation}
f(t) = \vector{a}\sin(\omega t + \psi) + \vector{b}\cos(\omega t + \psi)\label{eq:theoremab}
\end{equation}
where \vector{a} and \vector{b} are orthogonal vectors defining
the major and minor axes of the ellipse respectively
(that is $\vector{a}\perp\vector{b}$, and $\norm{\vector{a}}-\norm{\vector{b}}\ge0$),
and $\psi$ is a phase that is to be determined.

The vectors \vector{a} and \vector{b} are given by:
\begin{align}
\vector{a} &= \phantom{-}\vector{c}\cos\psi + \vector{s}\sin\psi\label{eq:a},\\
\vector{b} &=          - \vector{c}\sin\psi + \vector{s}\cos\psi\label{eq:b},
\end{align}
and $\psi$ is given by:
\begin{equation}
\frac{1}{2}\tan2\psi = \frac{\inner{\vector{c}}{\vector{s}}}{\norm{\vector{c}}-\norm{\vector{s}}}\label{eq:psi}.
\end{equation}
\end{theorem}
The vectors $\vector{a}$ and $\vector{b}$, and the phase $\psi$ are given here
in terms of the vectors $\vector{c}$ and $\vector{s}$ in Theorem \ref{thm:planar}
and hence are determined uniquely by the parameters of the original vectors in
\eqref{eqn:sum}, namely the directions of oscillation $\vector{n}_i$ and the
phases $\phi_i$.

Notice that in contrast to the formulation in Theorem \ref{thm:planar},
in this case \vector{a} and \vector{b} are orthogonal by definition,
and the sine and cosine are in quadrature because of the common phase
$\psi$.
The vectors \vector{a} and \vector{b} must be the major and minor axes of the ellipse
because the sine and cosine are in quadrature,
but it is not obvious that \vector{a} is the major axis,
however we demonstrate in the proof that this is always so.
\begin{proof}
We show how to construct \vector{a} and \vector{b} by finding $\psi$,
subject to the constraint $\vector{a}\perp\vector{b}$.
We then show that the result satisfies $\norm{\vector{a}}-\norm{\vector{b}}\ge0$.

Expand the cosine and sine
in \eqref{eq:theoremab}:
\begin{align}
f(t) & = 
\begin{gathered}[t]
\phantom{+}\,\vector{a}\left(\sin(\omega t)\cos\psi + \cos(\omega t)\sin\psi\right)\notag\\
         + \,\vector{b}\left(\cos(\omega t)\cos\psi - \sin(\omega t)\sin\psi\right)
\end{gathered}
\intertext{Regrouping the terms we obtain:}
&=
\begin{gathered}[t]
\phantom{+}\,(\vector{a}\cos\psi - \vector{b}\sin\psi)\sin(\omega t)\notag\\
         + \,(\vector{a}\sin\psi + \vector{b}\cos\psi)\cos(\omega t)
\end{gathered}
\end{align}
Comparing this result with \eqref{eqn:sum_cs}, we find that:
\begin{align*}
\vector{c} &= \vector{a}\cos\psi - \vector{b}\sin\psi\\
\vector{s} &= \vector{a}\sin\psi + \vector{b}\cos\psi
\end{align*}
We can write this in matrix-vector form as:
\[
\begin{pmatrix}
\cos\psi &          - \sin\psi\\
\sin\psi & \phantom{-}\cos\psi
\end{pmatrix}
\begin{pmatrix}
\vector{a}\\
\vector{b}
\end{pmatrix}
=
\begin{pmatrix}
\vector{c}\\
\vector{s}
\end{pmatrix}
\]
and we recognise the matrix as an orthogonal \emph{rotation matrix}%
\footnote{The rotation is not in the plane of the ellipse in signal space,
but in the space of the vectors
$\begin{pmatrix}\vector{a} & \vector{b}\end{pmatrix}^T$
and
$\begin{pmatrix}\vector{c} & \vector{s}\end{pmatrix}^T$.}
with unit determinant.
Hence we can express the orthogonal vectors \vector{a} and \vector{b}
in terms of the non-orthogonal vectors \vector{c} and \vector{s} as follows:
\[
\begin{pmatrix}
\vector{a}\\
\vector{b}
\end{pmatrix}
=
\begin{pmatrix}
\phantom{-}\cos\psi & \sin\psi\\
         - \sin\psi & \cos\psi
\end{pmatrix}
\begin{pmatrix}
\vector{c}\\
\vector{s}
\end{pmatrix}
\]
which gives \eqref{eq:a} and \eqref{eq:b}.

It remains to find $\psi$,
for which we make use of the properties of the inner product \cite{Renze:innerproduct},
in particular:
\begin{align}
\inner{\vector{u}+\vector{v}}{\vector{w}}&=\inner{\vector{u}}{\vector{w}} +
                                           \inner{\vector{v}}{\vector{w}}\label{eq:innermost}\\
\inner{\alpha\vector{u}}{\vector{v}}     &=\alpha\inner{\vector{u}}{\vector{v}}\label{eq:inner scalar}\\
\inner{\vector{u}}{\vector{v}}           &=\inner{\vector{v}}{\vector{u}}\label{eq:innercommute}
\end{align}
where \vector{u}, \vector{v} and \vector{w} are vectors, and $\alpha$ is a scalar.

Since we have defined \vector{a} and \vector{b} to be orthogonal,
we know that $\inner{\vector{a}}{\vector{b}}=0$,
and hence from \eqref{eq:a} and \eqref{eq:b} we have:
\begin{equation}\label{eq:orthogonality}
\inner{ \vector{c}\cos\psi + \vector{s}\sin\psi}%
      {-\vector{c}\sin\psi + \vector{s}\cos\psi}
= 0
\end{equation}
Applying \eqref{eq:innermost} to \eqref{eq:orthogonality}, we obtain:
\begin{align*}
&\inner{\vector{c}\cos\psi}{-\vector{c}\sin\psi + \vector{s}\cos\psi}\\
+
&\inner{\vector{s}\sin\psi}{-\vector{c}\sin\psi + \vector{s}\cos\psi} = 0
\end{align*}
and making use of \eqref{eq:innercommute}
we can then use \eqref{eq:innermost} again on each of these inner products, giving:
\begin{align*}
-&\inner{\vector{c}\cos\psi}{\vector{c}\sin\psi}
+ \inner{\vector{s}\sin\psi}{\vector{s}\cos\psi}\\
+&\inner{\vector{c}\cos\psi}{\vector{s}\cos\psi}
- \inner{\vector{s}\sin\psi}{\vector{c}\sin\psi} = 0
\end{align*}
Factoring out the scalars using \eqref{eq:inner scalar},
and noting that $\inner{\vector{u}}{\vector{u}}=\norm{\vector{u}}$:
\begin{gather*}
\sin\psi\cos\psi\left(\norm{\vector{s}}-\norm{\vector{c}}\right)\\
+
\left(\cos^2\psi-\sin^2\psi\right)\inner{\vector{c}}{\vector{s}} = 0
\end{gather*}
Finally, the double angle formulae
reduce this to:
\begin{equation*}
\frac{1}{2}\sin2\psi
\left(\norm{s}-\norm{c}\right) + \cos2\psi\inner{\vector{c}}{\vector{s}} = 0
\end{equation*}
Re-arranging, we obtain \eqref{eq:psi}, from which $\psi$ can be found from
\vector{c} and \vector{s}.

The second part of the proof shows that \vector{a} is always the major
axis of the ellipse by demonstrating that the sign of
$\norm{\vector{a}}-\norm{\vector{b}}$ is never negative.

Applying the cosine rule\footnote{In vector form: $\norm{\vector{u}} = \norm{\vector{v}} + \norm{\vector{w}} - 2\inner{\vector{v}}{\vector{w}}$, \cite{cosine_law}.} on \eqref{eq:a} we get:
\begin{align}
\norm{\vector{a}} &= \norm{\vector{c}}\cos^2\psi
                   + \norm{\vector{s}}\sin^2\psi + 2\inner{\vector{c}}{\vector{s}}\sin\psi\cos\psi\notag\\
                  &= \norm{\vector{c}}\cos^2\psi
                   + \norm{\vector{s}}\sin^2\psi + \inner{\vector{c}}{\vector{s}}\sin2\psi\notag
\intertext{and on \eqref{eq:b} we get:}
\norm{\vector{b}} &= \norm{\vector{s}}\cos^2\psi
                   + \norm{\vector{c}}\sin^2\psi - 2\inner{\vector{c}}{\vector{s}}\sin\psi\cos\psi\notag\\
                  &= \norm{\vector{s}}\cos^2\psi
                   + \norm{\vector{c}}\sin^2\psi - \inner{\vector{c}}{\vector{s}}\sin2\psi\notag
\end{align}
Taking the difference $\norm{\vector{a}}-\norm{\vector{b}}$ we obtain:
\begin{gather}
\left(\norm{\vector{c}}-\norm{\vector{s}}\right)\left(\cos^2\psi-\sin^2\psi\right)
+ 2\inner{\vector{c}}{\vector{s}}\sin2\psi\notag\\
=\left(\norm{\vector{c}}-\norm{\vector{s}}\right)\cos2\psi
+ 2\inner{\vector{c}}{\vector{s}}\sin2\psi\label{eq:diff}
\end{gather}
From \eqref{eq:psi}, we can obtain an expression for the inner product of \vector{c} and \vector{s}
in terms of $\psi$ and the norms of \vector{c} and \vector{s}:
\begin{align}
\inner{\vector{c}}{\vector{s}}
&= \left(\norm{\vector{c}}-\norm{\vector{s}}\right)\frac{1}{2}\tan 2\psi\notag\\
&= \left(\norm{\vector{c}}-\norm{\vector{s}}\right)\frac{\sin 2\psi}{2\cos 2\psi}\notag
\end{align}
Substituting this result into \eqref{eq:diff} we find:
\begin{align}
\norm{\vector{a}}-\norm{\vector{b}} &=
\left(\norm{\vector{c}}-\norm{\vector{s}}\right)
\left(\cos2\psi + \frac{\sin^2 2\psi}{\cos2\psi}\right)\notag\\
&=
\left(\norm{\vector{c}}-\norm{\vector{s}}\right)
\left(\frac{\cos^2 2\psi+\sin^2 2\psi}{\cos2\psi} \right)\notag\\
&=\left(\norm{\vector{c}}-\norm{\vector{s}}\right)/\cos2\psi\label{eq:absign}
\end{align}
The sign of the right-hand side is determined by the sign of $\norm{\vector{c}}-\norm{\vector{s}}$,
which is arbitrary\footnote{Of course it depends on the directions of the vectors $\vector{n}_i$,
and the phases $\phi_i$ in \eqref{eqn:sum}, but what we mean here is that either sign is possible,
depending on these quantities.}
depending on the relative magnitudes of the two vectors \vector{c} and \vector{s};
and on the sign of $\cos2\psi$, which is not arbitrary,
since it depends on the sign of $\norm{\vector{c}}-\norm{\vector{s}}$
and the sign of \inner{\vector{c}}{\vector{s}}, which in turn depends on the angle
between \vector{c} and \vector{s}, but not the relative magnitudes of the two vectors.

Thus we have to consider four cases,
corresponding to the quadrants in which $2\psi$ lies according to \eqref{eq:psi}.
\begin{table}[t]
\caption{\label{tab:signs}Determination of the sign of $\cos2\psi$.}
\begin{center}
\begin{tabular}{c|c||l||c}
$\norm{\vector{c}}-\norm{\vector{s}}$
& $\inner{\vector{c}}{\vector{s}}$ & Quadrant & $\cos2\psi$\\
\hline
$+$ & $+$ & First & $+$\\
$-$ & $+$ & Second & $-$\\
$-$ & $-$ & Third & $-$\\
$+$ & $-$ & Fourth & $+$\\
\hline
\end{tabular}
\end{center}
\end{table}
Table \ref{tab:signs} shows the four possible cases for the signs of the two quantities
that determine $2\psi$, and the quadrants of the plane in which each case occurs.
Knowing the quadrant in which $2\psi$ lies gives us the sign of $\cos2\psi$, as shown.

Notice that the sign of $\cos2\psi$ is the same as the sign of $\norm{\vector{c}}-\norm{\vector{s}}$,
and hence the sign of the right-hand side in \eqref{eq:absign} is never negative.
Therefore the sign of $\norm{\vector{a}}-\norm{\vector{b}}$ is never negative,
and this shows that $\vector{a}$ is never the minor axis of the ellipse,
\ie \vector{a} is always the major axis (except when $\norm{\vector{a}}=\norm{\vector{b}}$
and there is no major axis).
\end{proof}
The result that \vector{a} is never the minor axis of the ellipse is somewhat surprising,
since the parameters of the sinusoids that are summed to produce the ellipse are arbitrary.
However, the result always produces a vector \vector{a} which is the major axis of the ellipse
(except of course in the case when the ellipse degenerates to a circle).

Note that computationally,
the angle $\psi$ should be computed using an \texttt{atan2} function,
not by actually dividing the inner product by the difference of the norms as given
in \eqref{eq:psi}, in order to obtain a result which takes account of the signs of both.
Since the \texttt{atan2} function will return an angle in any of the four quadrants,
it follows from the factor of two in \eqref{eq:psi} that the value of $\psi$
is in the right half-plane, \ie $-\pi/2\le\psi\le\pi/2$.

\section{Simulation example}

\begin{figure}
\centerline{\includegraphics[width=0.75\textwidth]{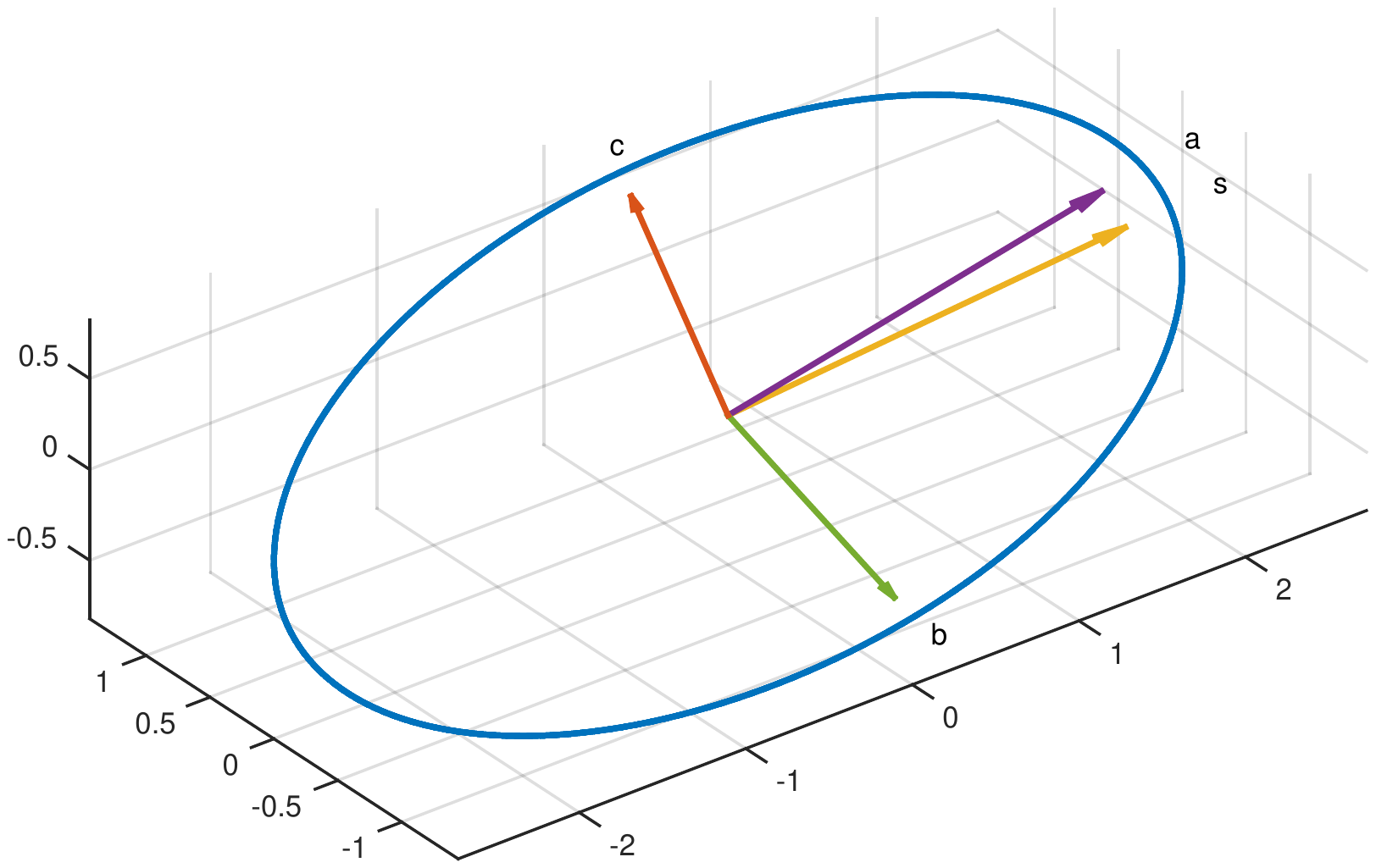}}
\caption{\label{fig:ellipse}Ellipse showing the four vectors: %
\vector{a} and \vector{b} (major and minor axes); %
\vector{c} and \vector{s} (cosine and sine vectors).
Vectors \vector{a} and \vector{b} are always orthogonal by definition.
There is no specific angular relationship between \vector{c} and \vector{s}.}
\end{figure}

\begin{figure}
\centerline{\includegraphics[width=0.75\textwidth]{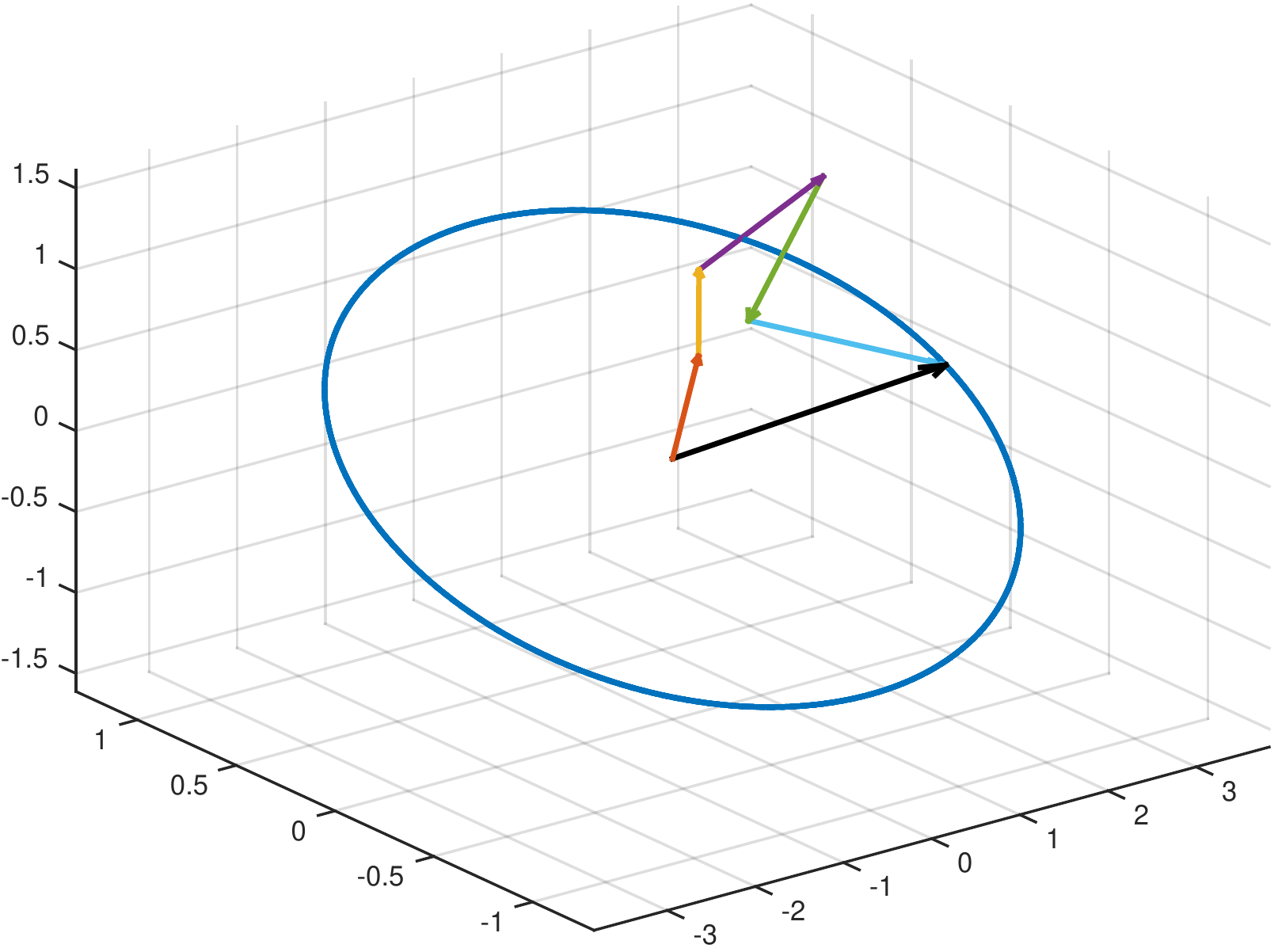}}
\caption{\label{fig:still}Still frame from the movie, showing the ellipse, five randomly chosen vectors
(in colour) and the resultant vector tracing out the ellipse (black).}
\end{figure}

We present a simulation example in the form of a \textsc{matlab} script that creates a movie/animation
in three dimensions, using five vectors with
randomly chosen directions and amplitudes, each scaled by a sine function of the
same frequency with randomly chosen phases, as in Theorem \ref{thm:planar}.
The choice of three dimensions was made simply because an easily understandable 3-D plot
is possible.
The \textsc{matlab} script (\texttt{ellipse.m}) is available for download as an ancillary file to this paper.

The first output of the simulation is a plot of the ellipse, showing the \vector{c} and
\vector{s} vectors of Theorem \ref{thm:planar} and the \vector{a} and \vector{b}
vectors of Theorem \ref{thm:major_minor}.
An example of this plot is shown in Figure \ref{fig:ellipse}.
The second output of the simulation is a movie/animation showing the five vectors and their resultant
sum, through one complete cycle of oscillation.
Figure \ref{fig:still} shows a still image from the animation.
Each one of the five vectors always points in the same direction (apart from reversal due to sign changes),
but their lengths scale sinusoidally, each with a different initial phase.
The resultant from summing the five vectors traces out the ellipse.

The simulation run presented is the result from the seventh run of the script after starting
\textsc{matlab} R2015a (or any version with the same random number generator).
Each run of the script will produce a different plot and movie ---
the ones presented were chosen because the plot shows clearly separated vectors
and the five vectors in the movie are nicely and clearly seen (the plot and movie
resulted from the same simulation run of course).

\section{Fourier analysis of vector signals}
\label{sec:ramifications}
In this section we consider the implications of the preceding material for
the analysis of vector signals into frequency components.
This discussion points the way to further work that can be done in the field of
Fourier transforms of vector signals and images,
particularly with regard to the interpretation of the Fourier domain representation
of a signal or image,
and in the construction of Fourier transforms for dimensions $N>2$.

The mathematical analysis presented in \S\,\ref{sec:ellipses} shows that the ellipse
resulting from the summation of an arbitrary number of sinusoids,
each of arbitrary amplitude and phase,
and each oscillating along an arbitrary direction in $N$-dimensional space,
is \emph{canonic},
that is only one ellipse can result from a summation of given sinusoids\footnote{The
converse is clearly not true since the same ellipse can fairly obviously result in more
than one way from a summation of sinusoids.}.
However, for $N>2$,
there is more than one way to define a Fourier transform,
and therefore the representation of the signal in the Fourier domain is not unique,
as it is in the case $N=2$ as represented,
for example,
by the classical complex Fourier transform.

A Fourier transform does not result in a representation directly in terms
of elliptical paths through the space of the vector values however,
even in the classical complex case.
Instead, the Fourier domain representation consists of values that
\emph{modify} a pair of exponentials
(with positive and negative frequencies)
that sum to produce the ellipse.
To make this clearer,
let us consider the classical complex Fourier transform in its discrete form
(the principles that concern us here work in the same way in the continuous case,
but are not so easily described).
A discrete Fourier transform pair may be written as:
\begin{align}
F[u] &= \frac{1}{\sqrt{M}}\sum_{m=0}^{M-1}f[m]\exp\left(          -\I 2\pi\frac{mu}{M}\right)\label{eq:dft}\\
f[m] &= \frac{1}{\sqrt{M}}\sum_{u=0}^{M-1}F[u]\exp\left(\phantom{-}\I 2\pi\frac{mu}{M}\right)\label{eq:idft}
\end{align}
where $\I$ is the imaginary root of $-1$,
$f[m]$ is a real or complex-valued discrete-time signal with $M$ samples,
$F[u]$ is complex valued, also with $M$ samples.

The complex exponential function has a complex value that traces a
unit circle in the complex plane,
as its argument varies from $0$ to $2\pi$.
In the discrete Fourier transform,
$M$ of these exponentials are summed,
each representing one possible frequency.
Half of the exponentials represent negative frequencies and half represent
positive frequencies, corresponding to the two possible senses of rotation.

Now consider how the coefficients $F[u]$ represent the signal in the Fourier domain.
Consider the `inverse' transform \eqref{eq:idft},
that reconstructs $f[m]$ from its frequency domain representation.
The frequency domain coefficients $F[u]$ occur in pairs,
corresponding to positive and negative frequencies,
(with the exception of the zero frequency coefficient $F[0]$,
and the Nyquist coefficient $F[M/2]$ --- which is absent if $M$ is odd).
Each pair of coefficients `scales' a pair of exponentials in both amplitude and phase.
The sum of the scaled and phase-shifted negative and positive frequency exponentials
is an ellipse of a given frequency as described in \S\,\ref{sec:ellipses}.
In the case where $f[m]$ is real, of course, the sum of the negative
and positive frequency exponentials is a degenerate ellipse oscillating along the
real axis (put another way, the imaginary components of the scaled and phase
shifted exponentials cancel out when summed).
The nature of the `scaling' in the complex case is that the amplitude of the
coefficient scales the amplitude of the exponential,
and the phase of the coefficient adds to the phase of the exponential.
(This is easily seen in polar form, of course.)

To construct Fourier transforms for the cases where $N>2$,
the nature of the `scaling' has to change,
because,
as we have seen in \S\,\ref{sec:ellipses},
we have to be able to construct an ellipse oriented in an arbitrary plane.
If we follow the example of the complex transform discussed above,
we need some way to construct an exponential with vector values in $N$ dimensions,
and we need to define multiplication to implement the `scaling' operation
such that we can modify the amplitude, phase,
\emph{and orientation}
of the circular path in $N$-dimensional space
represented by the exponential.
Two known ways to do this are:
\begin{itemize}
\item hypercomplex algebras
      (see, for example \cite{QCFTW:Brackx} for a historical overview of this topic);
\item matrix exponentials \cite{10.1016/j.amc.2012.06.055}.
\end{itemize}
We will discuss each of these in turn.

Much work has been done on \emph{hypercomplex} Fourier transforms using an
$N$-dimensional hypercomplex algebra.
In these algebras, which only exist in 
dimensions $N$ which are powers of two,
multiplication is often non-commutative
(where this is not the case,
other awkward properties occur,
such as the existence of divisors of zero).
In hypercomplex Fourier transforms,
elements of the algebra represent signal values,
and hypercomplex exponentials based on a square root of $-1$ in the algebra
generalise the concept of the complex exponential discussed above,
giving a circular path in an arbitrary plane defined by the orientation of
the square root of $-1$.
This relies on the generalisation of Euler's formula
$\exp\I\theta=\cos\theta+\I\sin\theta$, $\I^2=-1$,
which applies for complex numbers,
to a more general case $\exp\lozenge\theta=\cos\theta+\lozenge\sin\theta$
where $\lozenge$ represents a square root of $-1$, that is $\lozenge^2=-1$,
and the exponential function is defined for an
argument consisting of $\lozenge$ multiplied by a real scalar $\theta$.
In the case of hypercomplex algebras, $\lozenge$ would be an element of the algebra
(for examples, see \cite{10.1007/s00006-006-0005-8,QCFTW:7}),
but in general it could be something else, for example a matrix as in \cite{10.1016/j.amc.2012.06.055}.
The `scaling' concept then includes
(but often not in an easily understood manner)
a change of orientation of the exponential as discussed above,
in order to construct an arbitrarily-oriented ellipse in $N$ dimensions
by adding two (or more) exponentials rotating in opposite senses.
Non-commutative multiplication means that variants of the classical complex
Fourier transform can be constructed with the ordering of the exponential
and the signal reversed, giving slightly different results.
But more significantly, it is possible to define transforms with more than
one hypercomplex exponential
(for example,
one each side of the signal,
or two different exponentials on the same side,
or multiple different exponentials on each side).
For a discussion of many possibilities, see \cite{QCFTW:8}.
Interpretation of a transform in which exponentials are arranged on both
sides of the signal function is not simple.
Unlike our previous analysis where we considered the Fourier coefficients
as scale factors for the exponential,
it now makes more sense to consider the exponentials as \emph{operators}
on the Fourier coefficients,
an approach which we will not develop further here.

A second known way to construct a Fourier transform for signals with
$N$-dimensional vector values is to use matrix exponentials
\cite{10.1016/j.amc.2012.06.055}.
In this approach, the signal values may be represented by \emph{vectors} in the linear algebra sense
(that is a degenerate matrix with one row or column),
and the exponential by a matrix exponential with a matrix root of $-1$,
that is a matrix that squares to give a negated identity matrix.
The formulation is as given in \eqref{eq:dft} except that $\I$ is replaced by a matrix.
Unlike the case of hypercomplex algebras,
where the dimension $N$ is limited to powers of two,
real matrix roots of $-1$ exist for other even values of $N$
(but not for odd values of $N$ \cite[Theorem 2, p\,651]{10.1016/j.amc.2012.06.055}).
An intriguing aspect of this approach is that the matrix roots of $-1$ may be based
on matrix representations of hypercomplex algebras, in which case the transform is
numerically equivalent to a hypercomplex transform; or they may be arbitrary roots
of $-1$ that do not correspond to a hypercomplex algebra.
An example was given in \cite[\S\,5]{10.1016/j.amc.2012.06.055} where the matrix
exponential represents an elliptical path in 2-dimensions, thus making possible
a Fourier transform based inherently on ellipses rather than circles.
Further study of this topic is clearly merited, but it depends on first researching
the topic of matrix roots of $-1$ in a more general way than has so far been done.

\section{Conclusions}

The main result in this paper, presented in Theorem \ref{thm:planar}, is that any
vector signal in $N$-dimensions with a single frequency must oscillate along an
elliptical path, and therefore be confined to a plane.
This result has been demonstrated by synthesis: taking a sum of $K$ arbitrary signals
\emph{of the same frequency} $\nu$\,Hz
oscillating sinusoidally in arbitrary and different directions,
with arbitrary and different amplitudes and phases,
and showing that the resultant signal has values that traverse an elliptical path
in the $N$-dimensional signal space with a frequency of $\nu$ cycles of the path
per second.

The ramifications of this result are significant in the study and development of Fourier
transforms for $N$-dimensional vector signals, since each frequency component
of the signal, as analyzed by a Fourier transform, must oscillate in a plane.
This oscillation can always be produced by the superposition of two counter-rotating
exponentials, which shows that a one-sided Fourier transform is always sufficient.
More complicated formulations of Fourier transforms with multiple exponentials
will not yield a more complicated analysis of the signal in terms of frequency
components.

Clearly therefore, some future work is needed to revisit some of the ideas of
hypercomplex, Clifford, and other Fourier transforms applicable to vector signals, in
particular to consider how each type of transform represents the frequency content
of the signal in terms of the elliptical paths described in this paper.

\bibliographystyle{unsrtnat}
\bibliography{IEEEabrv,sangwine,maths,image_processing,paper}
\end{document}